\documentclass{amsart}

\usepackage{amssymb}
\usepackage{bm}
\usepackage{graphicx}
\usepackage[centertags]{amsmath}
\usepackage{amsfonts}
\usepackage{amsthm}
\usepackage{graphicx}

\usepackage[usenames,dvipsnames]{color}
\usepackage[colorlinks=true]{hyperref}
\hypersetup{urlcolor=blue, citecolor=green}

\numberwithin{equation}{section}

\newtheorem{thm}{Theorem}

\newtheorem{lem}[thm]{Lemma}

\newtheorem{prop}[thm]{Proposition}

\newtheorem{defn}[thm]{Definition}
\theoremstyle{definition}
\newtheorem{rem}{Remark}

\newcommand{\rr}{\mathbb{R}}

\newcommand{\ee}{\varepsilon}
\newcommand{\dl}{\delta} 
\newcommand{\ph}{\varphi}

\newcommand{\meg}{\geqslant}
\newcommand{\mik}{\leqslant}

\title[On The Novikov-Veselov Equation]{Well-Posedness And Ill-Posedness Results For The Novikov-Veselov Equation}
\author{Yannis Angelopoulos}
\begin{document}
\address{Department of Mathematics, University of Toronto, Toronto, On, Canada}
\email{yannis@math.toronto.edu}
\begin{abstract}
 In this paper we study the Novikov-Veselov equation and the related modified Novikov-Veselov equation in certain Sobolev spaces. We prove local well-posedness in $H^s (\rr^2)$ for $s > \frac{1}{2}$ for the Novikov-Veselov equation, and local well-posedness in $H^s (\rr^2)$ for $s > 1$ for the modified Novikov-Veselov equation. Finally we point out some ill-posedness issues for the Novikov-Veselov equation in the supercritical regime.
\end{abstract}
\maketitle
\tableofcontents
\section{Introduction}
The Novikov-Veselov equation was introduced by Novikov and Veselov in \cite{NV1}, \cite{NV2}. It has the following form:

\begin{equation}\label{1}
 \left\{\begin{aligned}
        \partial_t u + \partial^3 u +\bar{\partial}^3 u + NL_1 (u) + NL_2 (u) = 0 \\
        u(0,x,y) = \phi(x,y)\
        \end{aligned}
 \right.
\end{equation}
and as it is indicated by the arguments in the functions above, it is an equation posed on $\rr^2$. Moreover by convention, both the initial data $\phi : \rr^2 \rightarrow \rr$ and the solution $u : \rr \times \rr^2 \rightarrow \rr$ are taken to be real-valued. The operators $\partial$ and $\bar{\partial}$ are given by the formulas $\partial = \frac{1}{2} (\partial_x - i \partial_y )$, $\bar{\partial} = \frac{1}{2} (\partial_x + i \partial_y )$, and the nonlinear part is given by:
$$ NL_1 (u) = \frac{3}{4} \partial (u \bar{\partial}^{-1} \partial u), \quad \quad NL_2 (u) = \frac{3}{4} \bar{\partial} (u \partial^{-1} \bar{\partial} u) $$
This equation has the remarkable property of being completely integrable. The form of its linear and its nonlinear parts suggest similarities with Korteweg-de-Vries type equations. To our knowledge, it hasn't been proven so far that \eqref{1} possesses solutions for data in any ``reasonable'' space (where by ``reasonable'' here we mean a standard Sobolev space), although many results have been obtained in other directions (see \cite{P} and the references therein).

There is plenty of literature around this equation through different methods (the inverse scattering method for instance), and in different formulations (at non zero energy, in our case the energy is 0). The interested reader can look at \cite{K}, \cite{KN1}, \cite{KN2}, \cite{KN3}, \cite{G}, \cite{GM}, \cite{LMS}, \cite{LMSS}, \cite{N} for different results on a variety of problems concerning \eqref{1}.

For the Novikov-Veselov equation \eqref{1} we use the Fourier restriction norm method of Bourgain (see \cite{B}, and also \cite{B1} for a result on an equation in dimension 2) in order to prove local well-posedness for initial data in $H^s (\rr^2)$ where $s> \frac{1}{2}$. The way that this method is implemented here follows closely the lines of the preprint of Molinet and Pilod \cite{MP}, where a similar result is proved for the Zakharov-Kuznetsov equation whose linear part resembles the one of \eqref{1}. It is based on a bilinear estimate for high-low frequency interactions (the analogue of (3.16) of \cite{MP}) and a Strichartz-type estimate given by a harmonic analysis result of Carbery, Kenig and Ziesler \cite{CKZ}. 

A closely related equation to \eqref{1} is the modified Novikov-Veselov equation (we use here the formulation given by Perry in \cite{P}, see also \cite{Bo}):
\begin{equation}\label{mnv}
 \left\{\begin{aligned}
        \partial_t u + \partial^3 u +\bar{\partial}^3 u + mNL_1 (u) + mNL_2 (u) + mNL_3 (u) + mNL_4 (u) = 0 \\
        u(0,x,y) = \phi(x,y)\
        \end{aligned}
 \right.
\end{equation}
where 
$$ mNL_1 (u) = \frac{3}{4} \partial u \bar{\partial}^{-1} \partial (|u|^2 ), \quad mNL_2 (u) = \frac{3}{4} \bar{\partial} u \partial^{-1} \bar{\partial} (|u|^2 ) $$ $$ mNL_3 (u) = \frac{3}{4} u \bar{\partial}^{-1} [\partial (\bar{u} \partial u )], \quad mNL_4 (u) = \frac{3}{4} u \partial^{-1} [\bar{\partial} (\bar{u} \bar{\partial} u )] $$

For \eqref{mnv} we use the same techniques to obtain new well-posedness results. Instead of a bilinear estimate, we use two trilinear ones for different frequency interactions and we rely again on Strichartz estimates in order to prove the crucial trilinear $X^{s,b}$ estimate. 

It should be noted that for the modified Novikov-Veselov equation, the existence of global solutions was proven by Perry in \cite{P} for data in the space
$$ H^{2,1} (\rr^2) \cap L^1 (\rr^2) \mbox{ where } $$ $$ H^{m,n} (\rr^2) = \{ u \in L^2 (\rr^2) | (I-\Delta)^{m/2} u, (1 + |\cdot|)^n u(\cdot) \in L^2 (\rr^2) \} $$

Finally we prove that for data in $\dot{H}^{s} (\rr^2)$ where $s < -1$ it is impossible to prove the existence of solutions using the fixed-point method, no matter which subspace of $\dot{H}^{s} (\rr^2)$ you choose. Ill-posedness results of this type were first demonstrated by Bourgain (see \cite{B2} for example). Here we follow the lines of the article of Molinet, Saut and Tzvetkov \cite{MST1} (see also the nice and short expositions \cite{MST2} and \cite{Tz}, the first one by the same team of authors for the Benjamin-Ono equation, and the second one by Tzvetkov for the KdV equation -- which is in some sense more similar to ours), where a much stronger result was proven for the Kadomtsev-Petviashvili I equation.

\paragraph{\textbf{Overview of the article.}} In the following section (Section 2) we introduce some notational conventions that will be used throughout the article. 

In Section 3, we record the dispersive estimates and the Strichartz estimates that follow from them. We also state the important $H^{-1/4}_{x,y} \rightarrow L^4_{t,x,y}$ Strichartz-type inequality for the linear propagator that was mentioned before, which is derived from a result of Carbery, Kenig and Ziesler.

In Section 4, we prove the bilinear and trilinear estimates (following the related section of the paper of Molinet and Pilod \cite{MP}) that will be used in the proofs of the main well-posedness results of the article. These results will be proven in sections 5 and 6. 

To be more precise, in Section 5 we show the following well-posedness theorem for \eqref{1} where the fixed-point arguments takes place in the Bourgain-type $X^{s,b}$ spaces for the Novikov-Veselov equation.
\begin{thm}\label{wpnv}
 The Novikov-Veselov equation \eqref{1} is locally well-posed in $H^s (\rr^2)$ for any $s > \frac{1}{2}$.
\end{thm}

And in Section 6 we will prove the analogue for \eqref{mnv}, but with a different regularity threshold.
\begin{thm}\label{wpmnv}
 The modified Novikov-Veselov equation \eqref{mnv} is locally well-posed in $H^s (\rr^2)$ for any $s > 1$.
\end{thm}
In both theorems, by "locally well-posed in $H^s (\rr^2)$" we mean that there is some time $\dl = \dl (\| \phi \|_{H^s (\rr^2)} )$ such that there exists a unique solution to either \eqref{1} or \eqref{mnv} satisfying $u(0,x,y) = \phi(x,y)$ and
$$ u \in C([0,\dl]; H^s (\rr^2)) \cap X^{s,b}_{\dl} \mbox{ for some $b > 1/2$} $$
Moreover since we show this by a fixed-point argument, the data-to-solution map is smooth in a neighborhood of the initial data in $H^s (\rr^2)$, in any time interval of the form $[0,\dl']$ for $\dl' \in (0,\dl)$.

It should be noted that from the point of view of scaling (and hence, of criticality) there is plenty of room for improvement in both cases.

Finally in Section 7, we prove a result that is more or less expected, that a fixed-point argument can't give us a solution in the supercritical regime for \eqref{1}, following the exposition of Molinet, Saut and Tzvetkov in \cite{MST1}. The actual formulation of our result is the following one:
\begin{thm}\label{ipnv}
 Fix any $s,T \in \mathbb{R}$, $s<-1$, $T > 0$. Then there is no continuously embedded subspace $X_T$ of
 $C([0,T]; \dot{H}^s (\mathbb{R}^2 ))$ where the following inequalities hold:
 $$ \| e^{itNV} \phi \|_{X_T} \lesssim \| \phi \|_{H^s (\mathbb{R}^2 )} \quad \forall \phi \in H^s (\mathbb{R}^2 ), t\in [0,T] $$
 $$ \left\| \int_0^t e^{i(t-s)NV} [NL_1 (u)(s) + NL_2 (u)(s)] ds \right\|_{X_T} \lesssim \| u \|_{X_T}^2 \quad \forall t\in[0,T] $$
\end{thm}
Such estimates would be needed for a fixed-point argument, so in the end we conclude that for data in $\dot{H}^s (\rr^2)$, there is no proper subspace to run the contraction scheme for the Duhamel formula.

\paragraph{\textbf{Acknowledgments.}} This work is part of the PhD thesis research of the author at the University of Toronto. The author would like to thank Professors James Colliander and Peter Perry for their encouragement and for many interesting conversations about this work.

The author would also like to thank Daniel Egli and Arick Shao for reading parts of this paper and for the many useful discussions about it.

\section{Notation}
For a function $f$ we denote by $\hat{f}$ its Fourier transform in space and by $\check{f}$ the inverse Fourier transform, and by $\tilde{f}$ its space-time Fourier transform. We use the symbols $\| \cdot \|_{L^p}$ and $\| \cdot \|_{H^s}$ for the Lebesgue and Sobolev norms of a function, respectively. For two quantities $A$, $B$, we use the relation $A \lesssim B$ to indicate that there is some positive constant $C > 0$ such that $A \mik CB$, and we use the relation $A \approx B$ to indicate that there is a (possibly different) constant $C > 0$ such that $C^{-1} B \mik A \mik CB$.

Moreover we introduce spectral cut-offs in space and in space and time. We define a function $\chi \in \mathcal{S} (\rr)$ such that $supp(\hat{\chi}) \subset [-2,2]$ and $\hat{\chi} = 1$ in $[-1,1]$. First we define $\tilde{\chi}$ by $$\hat{\tilde{\chi}} (\xi) = \hat{\chi} (\xi) - \hat{\chi} (2\xi) $$
Now we define $\ph$ and $\psi$ as: 
$$ \hat{\ph} (\xi,\mu) = \hat{\tilde{\chi}} (|(\xi,\mu)|) \mbox{ and } \hat{\psi}(\tau, \xi,\mu) = \hat{\tilde{\chi}} \left(\tau - \frac{1}{4} \xi^3 + \frac{3}{4} \xi\mu^2 \right) $$
For $k\in \mathbb{N}$, $k\meg 1$ we define further the functions $\ph_k$ and $\psi_k$:
$$ \hat{\ph}_k (\xi,\mu) = \hat{\ph} \left( \frac{(\xi,\mu)}{2^k} \right) \mbox{ and } \hat{\psi}_k (\tau,\xi,\mu) = \hat{\psi}\left( \frac{(\tau, \xi,\mu)}{2^k} \right) $$
and finally we define also $\ph_0$ and $\psi_0$:
$$ \hat{\ph}_0(\xi,\mu) = \hat{\chi} (|(\xi,\mu)|) \mbox{ and } \hat{\psi}_0 (\tau,\xi,\mu) = \hat{\chi} \left(\tau -\frac{1}{4} \xi^3 - \frac{3}{4} \xi\mu^2 \right) $$
These cut-offs in frequency form partitions of unity ($\{ \hat{\ph}_k \}_{k=0}^{\infty}$ and $\{ \hat{\psi}_k \}_{k=0}^{\infty}$) and give rise to the following Littlewood-Paley operators:
$$\widehat{P_k f} (\xi,\mu) = \hat{\ph}_k (\xi,\mu) \hat{f} (\xi,\mu), \quad \quad \widetilde{Q_k u} (\tau,\xi,\mu) = \hat{\psi}_k (\tau,\xi,\mu) \tilde{u} (\tau,\xi,\mu) \mbox{ for $k\meg 0$} $$

\section{Dispersive and Strichartz Estimates}
A computation shows that 
\begin{displaymath}
\partial^3 + \bar{\partial}^3 = \frac{1}{4} \partial_{xxx}^3 - \frac{3}{4} \partial_{xyy}^3
\end{displaymath}
which turns \eqref{1} into the following equation:
\begin{equation}\label{2}
 \left\{\begin{aligned}
        \partial_t u + \frac{1}{4} \partial_{xxx}^3 - \frac{3}{4} \partial_{xyy}^3 + NL_1 (u) + NL_2 (u) = 0 \\
        u(0,x) = \phi(x)\
        \end{aligned}
 \right.
\end{equation}
Finally we investigate some of the dispersive properties of this equation. 
We consider first the linear part of the equation, taking the Fourier transform in space. This gives us the following:
\begin{displaymath}
\hat{u} (t, \xi , \mu ) = e^{it \left( \frac{1}{4} \xi^3 -\frac{3}{4} \xi \mu^2 \right)} \hat{\phi} (\xi , \mu)
\end{displaymath}
From this formula it becomes clear that we have
\begin{displaymath}
\| e^{itNV} \phi \|_{L^2_x} = \| \phi \|_{L^2_x}
\end{displaymath}
where $e^{itNV}$ is the propagator of the linear part of the Novikov-Veselov equation. On the other hand we can define the following measure in $\mathbb{R}^{2+1}$ based on the solution given by the Fourier inversion formula above:
\begin{displaymath}
\int_{\mathbb{R}^{2+1}} F(\xi, \mu, \tau) d\rho(\xi, \mu, \tau) = \int_{\mathbb{R}^2} F(\xi, \mu, P(\xi, \mu)) d\xi
\end{displaymath}
where we define $P(\xi, \mu) = \frac{1}{4} \xi^3 - \frac{3}{4} \xi \mu^2$. Computing the Hessian of $P(\xi,\mu)$ (denote it by $HP$) we can see that $detHP(\xi,\mu) = 0 $ only for $\xi = \mu = 0$, so away from 0 the hypersurface defined by this measure has non-vanishing Gaussian curvature. Following the usual proof technique of Strichartz estimates (see for example \cite{MuSc}) we can prove using the above remark for the measure $\rho $ that for any function $\phi = P_N \phi$ that is frequency localized in $\{ 1/2 \mik |(\xi, \mu)| \mik 2 \}$ (or in some other dyadic block) we have by the standard stationary phase theorem
\begin{displaymath}
\|  e^{itNV} P_N \phi \|_{L^{\infty}_x} \lesssim \dfrac{1}{\langle t \rangle} \| P_N \phi \|_{L^1_x}
\end{displaymath}
Scaling considerations for the linear equation show that from a solution $u$ we can consider another solution: $$ u_{\lambda} (t,x , y) = u \left( \frac{t}{\lambda^3}, \frac{x}{\lambda}, \frac{y}{\lambda} \right) $$

Using this,the usual $TT^{*}$ argument and the Hardy-Littlewood-Sobolev inequality we finally get an $L^2 \rightarrow L^p_t L^q_x$ estimate for $(p,q)$ satisfying $\frac{3}{p} + \frac{2}{q} = 1$, so that in the end we have for any function $\phi$ that
\begin{displaymath}
\| e^{itNV} \phi \|_{L^p_t L^q_x} \lesssim \| \phi \|_{L^2_x} \mbox{  where  } \frac{3}{p} + \frac{2}{q} = 1, \quad 3 < p \mik \infty, 2 \mik q < \infty
\end{displaymath}
Note that the diagonal Strichartz pair is the $L^5_{t,x}$ one. Also note that just by rescaling we have the following more general estimate:
$$ \| e^{itNV} \phi \|_{L^p_t L^q_x} \lesssim \| \phi \|_{\dot{H}^{\gamma}_x} \mbox{  where  } \frac{3}{p} + \frac{2}{q} = 1 - \gamma, \gamma \meg 0 $$
Again, we don't consider the case of endpoints for any $\gamma$.

Note also that the Strichartz admissibility condition would give us that there is an $\dot{H}^{-1/4}_{x,y} \rightarrow L^4_{t,x,y}$ mapping property for the group $e^{itNV}$. Of course this doesn't follow directly from the proof of Strichartz estimates (the admissibility condition doesn't normally hold for negative Sobolev spaces), but it can be established differently using the following result of Carbery, Kenig and Ziesler \cite{CKZ} -- see section 3 of the Molinet-Pilod work \cite{MP} for the formulation given here.

\begin{thm}
 Let $Q(\xi,\mu)$ be a homogeneous polynomial of degree $\meg 2$, and let $K_Q (\xi, \mu) = detHQ(\xi,\mu)$. Also let $Q(D)$ and $|K_Q (D)|^{1/8}$ be the multipliers associated to $Q(\xi,\mu)$ and $|K_Q (\xi,\mu)|^{1/8}$ respectively. Then for any $f\in L^2 (\mathbb{R}^2 )$ we have that
 $$ \| |K_Q (D)|^{1/8} e^{itQ} f \|_{L^4_{t,x,y}} \lesssim \| f \|_{L^2 (\mathbb{R}^2 )} $$
\end{thm}
This result applies for $K(\xi,\mu) = P(\xi,\mu)$. But in this case notice that we have 
$$ detHP(\xi,\mu) = -\frac{3}{4} (\xi^2 + \mu^2) $$
So $K(D)$ is actually $\Delta$ up to a constant. So for \eqref{2} we have the following estimate:
\begin{equation}\label{l4}
\| |D|^{1/4} e^{itNV} \phi \|_{L^4_{t,x,y}} \lesssim \| \phi \|_{L^2 (\mathbb{R}^2 )} 
\end{equation}
\begin{rem}
For the similar computation concerning the Zakharov-Kuznetsov equation, see again section 3 of \cite{MP}.
\end{rem}

\section{Multilinear Estimates}
In this section -- which follows closely the section on Bilinear Estimates of \cite{MP} -- we prove bilinear and trilinear estimates that will be needed in the proofs of bilinear and trilinear estimates in $X^{s,b}$ spaces. Such estimates can be viewed as refinements of the Strichartz inequalities when the functions involved interact in a specific way with respect to their frequency localizations. Here we follow \cite{MP}, and we write these estimates in an "$X^{s,b}$ manner". But they can be written also as bilinear and trilinear estimates for properly frequency localized linear solutions $e^{itNV} \cdot$ (by the transference principle for $X^{s,b}$ spaces). 
 
First let us make the remark, that since we are asking for refinements of the Strichartz estimates, we can't rely entirely on H\"{o}lder's inequality and we have to take into consideration the interactions of the functions involved. In the $X^{s,b}$ formulations of multilinear estimates, this can be understood better  by looking at the so called "resonant" function which is defined in the following way: we consider initially the functions
$$ w(\tau, \xi, \mu) = \tau - \frac{1}{4} \xi^3 + \frac{3}{4} \xi\mu^2, \quad w_1 (\tau_1, \xi_1, \mu_1) = w(\tau_1, \xi_1, \mu_1), $$ $$ w_2 (\tau, \tau_1, \xi, \xi_1, \mu, \mu_1) = w(\tau-\tau_1, \xi-\xi_1, \mu-\mu_1) $$
and then we take a certain difference of these three function to arrive at the definition of the "resonant" function
$$ R(\xi_1, \xi-\xi_1, \mu_1, \mu-\mu_1) = w -w_1 - w_2 = $$ $$ = \frac{3}{4} \xi_1 \xi (\xi-\xi_1) - \frac{3}{4} \xi_1 (\mu-\mu_1)^2 - \frac{3}{4} (\xi-\xi_1)\mu_1^2 - \frac{3}{2} \xi\mu_1 (\mu-\mu_1) $$
which is related of course to the $Q$ localizations that were introduced in the notational section.

Before moving into the actual estimates let us state (without proof) two basic facts that will be used in the upcoming proofs. The first one is a version of the mean value theorem.
\begin{thm}\label{ap1}
 Let $I, J \subset \rr$ be two intervals and $f : J \rightarrow \rr$ be a smooth function. Then the following holds:
 $$ |\{ x \in J | f(x) \in I \}| \mik \dfrac{|I|}{\inf_{y \in J} |f'(y)|} $$
\end{thm}
The second one is another measure theoretic tool and can be found in \cite{MST}.
\begin{lem}\label{ap2}
 In this lemma we use the notation $(\xi, \mu) \in \rr \times \rr$, i.e. the ``first'' axis corresponds to $\xi$ and the ``second'' one to $\mu$.  Let $J \subset \rr \times \rr$. Assume also that the projection in $\mu$ is contained in some set $I\subset \rr$ and that there exists some constant $C > 0$ such that $\forall \mu_0 \in I$ it holds that
 $$ |J \cap \{(\xi, \mu_0)\}| \mik C $$
 Then we have that
 $$ | J | \mik C|I| $$
\end{lem}

Now we can state and prove the bilinear estimate. 
\begin{lem}\label{lem1}
 Let $k_f , k_g $ be such that
 $$ k_f \meg 2, k_g \mik k_f - 2$$
 which can be seen as $k_g << k_f$.
 
 Then we have:
\begin{equation}\label{bl}
 \| P_{k_f} Q_{l_f} f P_{k_g} Q_{l_g} g \|_{L^2_{t,x,y}} \lesssim \frac{2^{\frac{k_g}{2}}}{2^{k_f}} 2^{\frac{l_f}{2}} 2^{\frac{l_g}{2}} \| P_{k_f} Q_{l_f} f \|_{L^2_{t,x,y}} \| P_{k_g} Q_{l_g} g \|_{L^2_{t,x,y}}
\end{equation}
\end{lem}
\begin{proof}
Let $(\tau_1 , \xi_1 , \mu_1 )$ denote the variables corresponding to $\widetilde{P_{k_f} f}$ and $(\tau-\tau_1 , \xi-\xi_1 , \mu-\mu_1 )$ the ones corresponding to $\widetilde{P_{k_g} g}$. By our assumptions we have that 
$$ |(\xi_1 , \mu_1 )| \approx 2^{k_f} >> 2^{k_g} \approx |(\xi-\xi_1 , \mu-\mu_1 )| $$
By applying successively Plancherel's identity, Young's inequality, Cauchy-Schwarz and Plancherel again we have:
$$ \| P_{k_f} Q_{l_f} f P_{k_g} Q_{l_g} g \|_{L^2_{t,x,y}} = \| \widetilde{P_{k_f} Q_{l_f} f} \ast \widetilde{P_{k_g} Q_{l_g} g} \|_{L^2_{\tau,\xi,\mu}} \mik $$ $$ \mik\sup_{\tau,\xi,\mu} |A_{\tau,\xi,\mu}|^{1/2} \| P_{k_f} Q_{l_f} f \|_{L^2_{t,x,y}} \| P_{k_g} Q_{l_g} g \|_{L^2_{t,x,y}} $$
where the set $A_{\tau,\xi,\mu}$ is defined as
$$ A_{\tau,\xi\mu} = \left\{ (\tau,\xi,\mu) | |(\xi_1 , \mu_1 )| \approx 2^{k_f}, |(\xi-\xi_1 ,\mu-\mu_1 )| \approx 2^{k_g}, \right. $$ $$ \left. |\tau_1 - P_1 | \approx 2^{l_f}, |\tau-\tau_1 - P_2 | \approx 2^{l_g} \right\} $$
where $P_1 (\xi_1 , \mu_1 ) = P(\xi_1 , \mu_1)$ and $P_2 (\xi,\xi_1, \mu,\mu_1) = P(\xi-\xi_1, \mu-\mu_1)$.

Applying the triangle inequality we get:
$$ |A_{\tau,\xi,\mu}| \mik \min(2^{l_f}, 2^{l_g}) |B_{\tau,\xi,\mu}| $$
where 
$$ B_{\tau,\xi,\mu} = \{ (\xi_1, \mu_1) |  |(\xi_1 , \mu_1 )| \approx 2^{k_f}, |(\xi-\xi_1 ,\mu-\mu_1 )| \approx 2^{k_g}, |\tau + P - R| \lesssim \max(2^{l_f}, 2^{l_g}) \} $$
for $R = R(\xi,\xi_1,\mu,\mu_1)$ the ``resonant'' function. 

Now we consider three different cases by taking into account the interaction between $\xi_1$ and $\mu_1$.
\\
\textbf{Subcase 1: $|\xi_1| >> |\mu_1|$} In this situation we follow step-by-step the proof of estimate (3.16) as it given in pages 8 and 9 of \cite{MP} (see the remark after the proof for the reasoning). 

We apply Theorem \ref{ap1} for the set $B_{\tau,\xi,\mu}$ where we fix $\mu_1$ (we call this $B_{\tau,\xi,\mu} (\mu_1 )$) after computing the following derivative:
$$| \partial_{\xi_1} R (\xi_1, \xi-\xi_1, \mu_1, \mu-\mu_1) | = \left| \frac{3}{4} (\xi_1^2 - \mu_1^2) - \frac{3}{4}[ (\xi - \xi_1 )^2 - (\mu - \mu_1 )^2] \right| $$ 
Taking into account that $|\xi_1| >> |\mu_1|$ we have:
$$ | \partial_{\xi_1} R (\xi_1, \xi-\xi_1, \mu_1, \mu-\mu_1) | \gtrsim 2^{2k_f} $$
This gives us that:
$$|B_{\tau,\xi,\mu} (\mu_1 )| \mik \dfrac{\max (2^{l_f}, 2^{l_g})}{2^{2k_f}} $$
Applying Lemma \ref{ap2} we further get that:
$$|B_{\tau,\xi,\mu}| \mik \dfrac{\max (2^{l_f}, 2^{l_g}) 2^{k_g}}{2^{2k_f}} $$
which in the end gives us that 
$$|A_{\tau,\xi,\mu}| \mik \dfrac{2^{k_g}}{2^{2k_f}}\max (2^{l_f}, 2^{l_g}) \min (2^{l_f}, 2^{l_g})  $$
\\
\textbf{Subcase 2: $|\xi_1| << |\mu_1|$} This reduces to Subcase1 since we have again the same bound 
$$ | \partial_{\xi_1} R (\xi_1, \xi-\xi_1, \mu_1, \mu-\mu_1) | \gtrsim 2^{2k_f} $$
\\ 
\textbf{Subcase 3: $|\xi_1| \approx |\mu_1|$} Now the argument used in Subcases 1 and 2 can't work, since we are considering the case where $|\partial_{\xi_1} R|$ is obviously no longer bounded below by $2^{2k_f}$. We compute first the derivative of $R$ with respect to $\mu_1$.
$$ | \partial_{\mu_1} R (\xi_1, \xi-\xi_1, \mu_1, \mu-\mu_1) | = \left| -\frac{3}{2} (\xi - \xi_1 )\mu_1 + \frac{3}{2} \xi_1 (\mu-\mu_1) - \frac{3}{2} \xi \mu + 3 \xi \mu_1 \right| \Rightarrow  $$ 
\begin{equation}\label{mu}
\begin{aligned}
\Rightarrow | \partial_{\mu_1} R (\xi_1, \xi-\xi_1, \mu_1, \mu-\mu_1) | = \\ = \left| \frac{3}{2} \xi \mu_1 - \frac{3}{2} \xi (\mu-\mu_1 )- \frac{3}{2} (\xi - \xi_1 )\mu_1 + \frac{3}{2} \xi_1 (\mu-\mu_1) \right| 
\end{aligned}
\end{equation}
Since $|(\xi_1 , \mu_1 )| \approx 2^{k_f}$ and $|\xi_1| \approx |\mu_1|$, this implies that
\begin{equation}\label{xi}
 |\xi_1| \approx |\mu_1| \approx 2^{k_f} 
\end{equation}
Moreover since $k_f >> k_g$ and as $|(\xi-\xi_1 , \mu-\mu_1 )| \approx 2^{k_g}$, we have by \eqref{xi} that
\begin{equation}\label{xi2}
|\xi| \approx |\mu| \approx 2^{k_f} 
\end{equation}
Then, going back to \eqref{mu}, we observe that by \eqref{xi} and \eqref{xi2}, the first term is of order $2^{2k_f}$ while everything else is just of order $2^{k_g}$, so we conclude in this situation that we have the following estimate:
$$  | \partial_{\mu_1} R (\xi_1, \xi-\xi_1, \mu_1, \mu-\mu_1) | \gtrsim 2^{2k_f} $$
We prove now the required estimate by following again the proof given in \cite{MP}, just by interchanging the roles of $\xi_1$ and $\mu_1$. To be a bit more specific, in this situation we fix $\xi_1$ in $B$, i.e. we consider the set $B_{\tau,\xi,\mu}(\xi_1)$ and repeat the analysis of Subcase 1.
\end{proof}
\begin{rem}
Note that in Subcases 1 and 2 of Lemma \ref{lem1} the estimate is identical to the similar situation in the Zakharov-Kuznetsov equation. The partial derivative of the "resonant" function with respect to $\xi_1$ in that case has the form:
$$\left| \frac{3}{4} (\xi_1^2 + \mu_1^2) - \frac{3}{4}[ (\xi - \xi_1 )^2 + (\mu - \mu_1 )^2] \right| $$
and one can see that this gives us the desired bound for frequencies of different sizes.
\end{rem}
Now we turn to the trilinear estimates that we will need for the modified Novikov-Veselov equation \eqref{mnv}. Again we consider only specific frequency interactions. These trilinear estimates are actually based on the bilinear estimate (Lemma \ref{lem1}) that we just showed.
\begin{lem}\label{lem2}
1) Consider dyadic numbers $k_f, k_g, k_h$ with the property
$$ k_f \meg 2, k_g \meg k_f + 2, k_f - 1 \mik k_h \mik k_f + 1 $$
Then we have the estimate
$$\| P_{k_f} Q_{l_f} f P_{k_g} Q_{l_g} g P_{k_h} Q_{l_h} h \|_{L^2_{t,x,y}} \lesssim $$ $$ \lesssim \dfrac{2^{\frac{3k_f}{2}}}{2^{k_g}} 2^{\frac{l_f}{2}} 2^{\frac{l_g}{2}} 2^{\frac{l_h}{2}} \| P_{k_f} Q_{l_f} f \|_{L^2_{t,x,y}} \| P_{k_g} Q_{l_g} g \|_{L^2_{t,x,y}} \| P_{k_h} Q_{l_h} h \|_{L^2_{t,x,y}} $$
2) Consider dyadic numbers $k_f, k_g, k_h$ with the property
$$ k_f \meg 2, k_g \mik k_f - 2, k_f - 1 \mik k_h \mik k_f + 1 $$
Then we have the estimate
$$\| P_{k_f} Q_{l_f} f P_{k_g} Q_{l_g} g P_{k_h} Q_{l_h} h \|_{L^2_{t,x,y}} \lesssim $$ $$ \lesssim  2^{\frac{k_g}{2}} 2^{\frac{l_f}{2}} 2^{\frac{l_g}{2}} 2^{\frac{l_h}{2}} \| P_{k_f} Q_{l_f} f \|_{L^2_{t,x,y}} \| P_{k_g} Q_{l_g} g \|_{L^2_{t,x,y}} \| P_{k_h} Q_{l_h} h \|_{L^2_{t,x,y}} $$
\end{lem}
\begin{proof}
 1) Our condition on the frequencies tells us that $f$ and $h$ are localized at roughly the same level, while $g$ is localized at a bigger one. We don't assume any condition for the localizations $l_f, l_g$ and $l_h$.
 
 First we apply Plancherel and we get
 $$ \| P_{k_f} Q_{l_f} f P_{k_g} Q_{l_g} g P_{k_h} Q_{l_h} h \|_{L^2_{t,x,y}} = \| \widetilde{P_{k_f} Q_{l_f} f} \ast \widetilde{P_{k_g} Q_{l_g} g} \ast \widetilde{P_{k_h} Q_{l_h} h} \|_{L^2_{\tau,\xi,\mu}} $$
Now we apply Young's inequality (and Plancherel for $P_{k_f} Q_{l_f} f P_{k_g} Q_{l_g} g$) and we have as in Lemma \ref{lem1} that the above quantity is bounded by the following
$$ \lesssim \sup_{\tau,\xi,\mu} |C_{\tau,\xi,\mu}|^{1/2} \| \widetilde{P_{k_h} Q_{l_h} h} \|_{L^2_{\tau,\xi,\mu}} \| P_{k_f} Q_{l_f} f P_{k_g} Q_{l_g} g\|_{L^2_{t,x,y}} $$
where 
$$C_{\tau,\xi,\mu} = \left\{ (\tau,\xi,\mu) | |(\xi_1 , \mu_1 )| \approx 2^{k_h}, |(\xi-\xi_1 ,\mu-\mu_1 )| \approx 2^{k_f} + 2^{k_g}, \right. $$ $$\left. |\tau_1 - P_1 | \approx 2^{l_h}, |\tau-\tau_1 - P_2 | \approx 2^{l_f} + 2^{l_g} \right\} $$
This is a similar definition to the one of $A_{\tau,\xi,\mu}$ in Lemma \ref{lem1}. But just note here that the variables $(\xi-\xi_1,\mu-\mu_1)$ correspond to the convolution $\widetilde{P_{k_f} Q_{l_f} f} \ast \widetilde{P_{k_g} Q_{l_g} g}$ and that's where the sizes of $|(\xi-\xi_1,\mu-\mu_1)|$ and $|\tau - \tau_1 - P_2 |$ come from. 

At this point we apply Lemma \ref{lem1} to the second term and Plancherel to the first:
$$ \sup_{\tau,\xi,\mu} |C_{\tau,\xi,\mu}|^{1/2} \frac{2^{\frac{k_f}{2}}}{2^{k_g}} 2^{\frac{l_f}{2}} 2^{\frac{l_h}{2}} \| P_{k_f} Q_{l_f} f \|_{L^2_{t,x,y}} \| P_{k_g} Q_{l_g} g \|_{L^2_{t,x,y}} \| P_{k_h} Q_{l_h} h \|_{L^2_{t,x,y}} $$
We use a crude bound for the set $C_{\tau,\xi,\mu}$ and we have
$$ |C_{\tau,\xi,\mu}|^{1/2} \lesssim \min(2^{\frac{l_h}{2}}, \max(2^{\frac{k_f}{2}}, 2^{\frac{k_f}{2}})) 2^{k_h} $$
which in the end gives us the desired inequality
$$ \| P_{k_f} Q_{l_f} f P_{k_g} Q_{l_g} g P_{k_h} Q_{l_h} h \|_{L^2_{t,x,y}} \lesssim $$ $$ \lesssim \dfrac{2^{\frac{3k_f}{2}}}{2^{k_g}} 2^{\frac{l_f}{2}} 2^{\frac{l_g}{2}} \min(2^{\frac{l_h}{2}}, \max(2^{\frac{k_f}{2}}, 2^{\frac{k_f}{2}})) \times $$ $$ \times \| P_{k_f} Q_{l_f} f \|_{L^2_{t,x,y}} \| P_{k_g} Q_{l_g} g \|_{L^2_{t,x,y}} \| P_{k_h} Q_{l_h} h \|_{L^2_{t,x,y}} \lesssim $$ $$ \lesssim  \dfrac{2^{\frac{3k_f}{2}}}{2^{k_g}} 2^{\frac{l_f}{2}} 2^{\frac{l_g}{2}} 2^{\frac{l_h}{2}} \| P_{k_f} Q_{l_f} f \|_{L^2_{t,x,y}} \| P_{k_g} Q_{l_g} g \|_{L^2_{t,x,y}} \| P_{k_h} Q_{l_h} h \|_{L^2_{t,x,y}} $$

The last inequality is just a crude bound, if $l_h$ is smaller than $l_f$ or $l_g$ we just have equality, if $l_h$ is the largest of all three then we have actually a better estimate, but we don't really need it for our purposes.

2) In this situation $f$ and $h$ are frequency localized at roughly the same level, while $g$ is frequency localized at a smaller one. 

For the proof of the estimate we follow the exact same method as in part 1 of this lemma, and again we will consider one case for the space-time localizations, namely that $l_h \mik l_f, l_g$. We estimate $C_{\tau,\xi,\mu}$ in the same way, but now we take into consideration that 
$$\max(k_f, k_g, k_h) = k_f \mbox{ or } k_h \mbox{ (and $k_f \approx k_h$)} $$
which means that 
$$ |C_{\tau,\xi,\mu}|^{1/2} \| P_{k_f} Q_{l_f} f P_{k_g} Q_{l_g} g \|_{L^2_{t,x,y}} \| P_{k_h} Q_{l_h} h \|_{L^2_{t,x,y}} \lesssim $$ $$ \lesssim \dfrac{2^{\frac{k_g}{2}} 2^{k_h}}{2^{k_f}} 2^{\frac{l_f}{2}} 2^{\frac{l_g}{2}} \min(2^{\frac{l_h}{2}}, \max(2^{\frac{k_f}{2}}, 2^{\frac{k_f}{2}})) \times $$ $$ \times \| P_{k_f} Q_{l_f} f \|_{L^2_{t,x,y}} \| P_{k_g} Q_{l_g} g \|_{L^2_{t,x,y}} \| P_{k_h} Q_{l_h} h \|_{L^2_{t,x,y}} \approx $$ $$ \approx 2^{\frac{k_g}{2}} 2^{\frac{l_f}{2}} 2^{\frac{l_g}{2}} \min(2^{\frac{l_h}{2}}, \max(2^{\frac{k_f}{2}}, 2^{\frac{k_f}{2}})) \times $$ $$ \times \| P_{k_f} Q_{l_f}  f \|_{L^2_{t,x,y}} \| P_{k_g} Q_{l_g}  g \|_{L^2_{t,x,y}} \| P_{k_h} Q_{l_h} h \|_{L^2_{t,x,y}} \lesssim $$ $$ \lesssim 2^{\frac{k_g}{2}} 2^{\frac{l_f}{2}} 2^{\frac{l_g}{2}} 2^{\frac{l_h}{2}} \| P_{k_f} Q_{l_f}  f \|_{L^2_{t,x,y}} \| P_{k_g} Q_{l_g}  g \|_{L^2_{t,x,y}} \| P_{k_h} Q_{l_h} h \|_{L^2_{t,x,y}} $$
The last inequality follows as in 1).

\end{proof}

\section{Well-Posedness for the Novikov-Veselov equation}
In this section our goal is to prove Theorem \ref{wpnv}. We introduce first the Bourgain-type spaces where the contraction scheme will be implemented. The $X^{s,b}$ space associated to the Novikov-Veselov equation is the following:
\begin{displaymath}
X^{s,b} = \left\{ u \in L^2_{t,x} | \langle (\xi, \mu) \rangle^s \left\langle \tau - \frac{1}{4}\xi^3 + \frac{3}{4}\xi \mu^2 \right\rangle^b \tilde{u} (\tau, \xi, \mu) \in L^2_{\tau, \xi, \mu} \right\} \mbox{ with norm }
\end{displaymath}
\begin{displaymath}
\| u \|_{X^{s,b}} = \left\| \langle (\xi, \mu) \rangle^s \left\langle \tau - \frac{1}{4} \xi^3 + \frac{3}{4} \xi \mu^2 \right\rangle^b \tilde{u} (\tau, \xi, \mu) \right\|_{L^2_{\tau, \xi, \mu}}
\end{displaymath}
\subsection{Bilinear Estimates}
In order to be able to apply Bourgain's machinery (see for example \cite{B}, \cite{T}), we will need the following proposition.
\begin{prop}\label{propbl}
 The following inequality holds true: $$ \| NL_1 (u,v) \|_{X^{s,-1/2+2\ee}} \lesssim \| u \|_{X^{s,1/2+\ee}} \| v \|_{X^{s,1/2+\ee}} $$ for any $s >  \frac{1}{2}$ and any $\ee > 0$. 
\end{prop}
\begin{rem} 
The proof of this proposition applies in the same way for $NL_2 (u,v) = \frac{3}{4} \bar{\partial} (u \partial^{-1} \bar{\partial} u)$. This shows that for the nonlinear part of the Novikov-Veselov equations, the following bilinear estimate holds: 
\begin{equation}\label{nl}
\| NL(u) \|_{X^{s,-1/2+2\ee}} = \| NL_1 (u,u) + NL_2 (u,u) \|_{X^{s,-1/2+2\ee}} \lesssim \| u \|_{X^{s,1/2+\ee}}^2 
\end{equation}
 
\end{rem}
\begin{proof}
We will use the functions $w, w_1, w_2$ that were defined in the previous section.

A standard duality argument transforms the estimate as shown below: 
$$ \| NL_1 (u,v) \|_{X^{s,-1/2+2\ee}} \lesssim \| u \|_{X^{s,1/2+\ee}} \| v \|_{X^{s,1/2+\ee}} \Rightarrow $$ $$ \Rightarrow \sup_{\| h' \|_{X^{-s, 1/2-2\ee}} =1} \left| \int_{\mathbb{R}^{2+1}} h' (t,x,y) NL(u,v) (t,x,y) dt dx dy \right| \lesssim $$ $$ \lesssim \| u \|_{X^{s,1/2+\ee}} \| v \|_{X^{s,1/2+\ee}} $$
We want to eliminate the $X^{s,b}$ norms and be left only with $L^2$ ones. First we apply Plancherel to the left-hand side and we rewrite the right-hand side:
$$ \sup_{\| \tilde{h} \|_{X^{-s, 1/2-2\ee}} =1} \left| \int_{\mathbb{R}^{2+1}} \tilde{h'} (\tau, \xi, \mu) \widetilde{NL}(u,v) (\tau, \xi, \mu) d\tau d\xi d\mu \right| \lesssim $$ $$ \lesssim \| \langle (\xi , \mu) \rangle^s \langle w \rangle^{1/2+\ee} \tilde{u} \|_{L^2_{\tau, \xi,\mu}} \| \langle (\xi , \mu) \rangle^s \langle w \rangle^{1/2+\ee} \tilde{v} \|_{L^2_{\tau, \xi,\mu}} $$

We make the following definitions: $$ \tilde{f} (\tau_1,\xi_1,\mu_1) = |\tilde{u} (\tau_1,\xi_1,\mu_1) \langle w_1 \rangle^{1/2+\ee} \langle(\xi_1,\mu_1)\rangle^s |, $$ $$ \tilde{g} (\tau-\tau_1, \xi-\xi_1, \mu-\mu_1) = |\tilde{v}(\tau-\tau_1, \xi-\xi_1, \mu-\mu_1) \langle w_2 \rangle^{1/2+\ee} \langle(\xi-\xi_1, \mu-\mu_1) \rangle^s $$ $$ \tilde{h} (\tau, \xi, \mu) = |\tilde{h'} (\tau,\xi,\mu) \langle w \rangle^{1/2-2\ee} \langle(\xi,\mu) \rangle^{-s}| $$

Computing the convolutions coming from $\widetilde{NL}(u,v)$ we restate our estimate once more:
$$ I = \int_{\mathbb{R}^6} K \tilde{f} (\tau_1,\xi_1,\mu_1) \tilde{g} (\tau-\tau_1, \xi-\xi_1, \mu-\mu_1) \tilde{h} (\tau,\xi,\mu) d\tau d\tau_1 d\xi d\xi_1 d\mu d\mu_1 \lesssim $$ $$ \lesssim \| f \|_{L^2_{\tau,\xi,\mu}} \| g \|_{L^2_{\tau,\xi,\mu}} \| h \|_{L^2_{\tau,\xi,\mu}} $$ where the function $K$ is defined as $$ K(\tau,\tau_1,\xi,\xi_1,\mu.\mu_1) = $$ $$ \dfrac{|i\xi+\mu||i(\xi-\xi_1) + (\mu-\mu_1)|\langle (\xi,\mu) \rangle^s}{|-i(\xi-\xi_1) + (\mu-\mu_1)| \langle w \rangle^{1/2-\ee} \langle (\xi_1,\mu_1) \rangle^s \langle w_1 \rangle^{1/2+\ee} \langle (\xi-\xi_1, \mu-\mu_1 ) \rangle^s \langle w_2 \rangle^{1/2+\ee} } = $$ $$ = \dfrac{|(\xi,\mu)|\langle (\xi,\mu) \rangle^s}{\langle w \rangle^{1/2-\ee} \langle (\xi_1,\mu_1) \rangle^s \langle w_1 \rangle^{1/2+\ee} \langle (\xi-\xi_1, \mu-\mu_1 ) \rangle^s \langle w_2 \rangle^{1/2+\ee} } $$ 
In the rest of the proof, we'll localize $f, g, h$ in certain frequencies (i.e. restrict the range of $(\xi,\mu), (\xi_1 , \mu_1)$ and $(\xi-\xi_1 , \mu-\mu_1)$). When this happens,  we'll call the frequencies corresponding to $\tilde{f}$ by $k_f$, and we use the same notation for $\tilde{g}, \tilde{h}$. Then $I$ restricted to these frequencies will be called $I_{k_f , k_g , k_h}$, specifically we'll have:
$$ I_{k_f , k_g , k_h} = \int_{\rr^6} K_{k_f , k_g , k_h}  \widetilde{f} (\tau_1,\xi_1,\mu_1) \widetilde{g} (\tau-\tau_1, \xi-\xi_1, \mu-\mu_1) \widetilde{h} (\tau,\xi,\mu) d\tau d\tau_1 d\xi d\xi_1 d\mu d\mu_1 $$

We'll break the proof into several cases dealing with the interactions between different frequencies (with respect to space and not time).

\paragraph{\textbf{Case 1: Low-Low-Low Frequencies}}
First we consider the case where $\tilde{f}, \tilde{g}, \tilde{h}$ have their $(\xi, \mu)$ supports in approximately the same region, a dyadic shell of size $2^k$ for $k$ small. We consider the case where
$$k_f , k_g , k_h \mik 1 $$
This case is trivial for $s \meg 1$ since in such a situation we have that $K \lesssim 1$ and we use Cauchy-Schwarz to throw $h$ in $L^2$ and then use the $X^{s,b}$ version of \eqref{l4} which reads as:
\begin{equation}\label{x4}
 \left\| \mathcal{F}^{-1}_{t,x,y} \left( \frac{|\xi|^{1/4} \tilde{f}}{\langle w \rangle^{1/2+\ee}} \right) \right\|_{L^4_{t,x,y}} \lesssim \| f \|_{L^2_{t,x,y}}
\end{equation}
for $f$ and $g$. Notice that this gives us a gain of $1/4$ of derivative, so the range of $s$ can be improved. In more detail (and using that $\langle w \rangle^{1/2+\ee} \gtrsim 1$ and similarly for $w_1$, $w_2$) we have:
$$ I_{LLL} = \sum_{k_f , k_g , k_h \mik 1} I_{k_f , k_g , k_h } \lesssim $$ $$ \lesssim  \sum_{k_f , k_g , k_h \mik 1} \dfrac{2^{k_h (s+1)}}{2^{k_f s} 2^{k_g s}} \| P_{k_h} h \|_{L^2_{t,x,y}}  \times $$ $$ \times \frac{2^{\frac{k_f}{4}}}{2^{\frac{k_f}{4}}} \left\| \mathcal{F}^{-1}_{t,x,y}\left( \frac{\tilde{f}}{\langle w \rangle^{1/2+\ee}} \right) \right\|_{L^4_{t,x,y}} \frac{2^{\frac{k_f}{4}}}{2^{\frac{k_f}{4}}}\left\| \mathcal{F}^{-1}_{t,x,y}\left( \frac{\tilde{f}}{\langle w \rangle^{1/2+\ee}} \right) \right\|_{L^4_{t,x,y}} \lesssim $$ $$ \lesssim \sum_{k_f , k_g , k_h \mik 1} \dfrac{2^{k_h (s+1)}}{2^{k_f (s+1/4)} 2^{k_g (s+1/4)}} \| P_{k_h} h \|_{L^2_{t,x,y}} \| P_{k_f} f \|_{L^2_{t,x,y}} \| P_{k_g} g \|_{L^2_{t,x,y}} $$
Now since $k_f \approx k_g \approx k_h$, we'll be able to apply Cauchy-Schwarz in all the frequencies for 
$$ s+1 - 2s -\frac{1}{2} < 0 \Rightarrow s > \frac{1}{2} $$
as follows for some $\ee' > 0$:
$$ \sum_{k_f \approx k_g \approx k_h} 2^{-\ee' k_g} \| P_{k_h} h \|_{L^2_{t,x,y}} \| P_{k_f} f \|_{L^2_{t,x,y}} \| P_{k_g} g \|_{L^2_{t,x,y}}  \lesssim $$ $$\lesssim \left(\sum_{k_g} \| P_{k_g} g \|_{L^2_{t,x,y}}^2 \right)^{1/2} \sum_{k_f \approx k_h} \| P_{k_h} h \|_{L^2_{t,x,y}} \| P_{k_f} f \|_{L^2_{t,x,y}} \left( \sum_{k_g} 2^{-2\ee' k_g} \right)^{1/2} \lesssim $$ $$ \lesssim \| g \|_{L^2_{t,x,y}} \sum_k \| P_k h \|_{L^2_{t,x,y}} \| P_k f \|_{L^2_{t,x,y}} \lesssim $$ $$ \lesssim \| g \|_{L^2_{t,x,y}}  \left( \sum_k \| P_k f \|_{L^2_{t,x,y}}^2 \right)^{1/2}  \left( \sum_k \| P_k h \|_{L^2_{t,x,y}}^2 \right)^{1/2} \lesssim $$ $$ \lesssim \| g \|_{L^2_{t,x,y}} \| f \|_{L^2_{t,x,y}} \| h \|_{L^2_{t,x,y}} $$

\paragraph{\textbf{Case 2: High-High-High Frequencies}}
Again all three frequencies are comparable. But now we consider the following set:
$$ \{ k_f , k_g , k_h | k_f , k_g \meg 2, k_g - 1 \mik k_f \mik k_g + 1, k_f - 1 \mik k_h \mik k_f + 1, k_g - 1 \mik k_h \mik k_g + 1 \} $$
which we call $J$.

We note that for each triplet $(k_f , k_g , k_h )$ we can follow the exact same proof as in Case1. The fact that all three frequencies are roughly the same allows us to add them in the same way, so again for $s > \frac{1}{2}$ (and just $s meg 1$ as stated in the theorem) we have the desired bound for 
$$ I_{HHH} = \sum_{J} I_{k_f , k_g , k_h } $$

\paragraph{\textbf{Case 3: High-High-Low Frequencies}}
We consider the interaction between high frequencies for $f$ and $g$ and low ones for $h$ (which is possible by the convolution in $I$). The set of frequencies is the following:
$$ \{ k_f , k_g , k_h | k_f \meg 2, k_h \mik k_f - 2, k_f - 1 \mik k_g \mik k_f + 1 \} = J $$
There is a symmetric case where the roles of $k_f$ and $k_g$ are interchanged, but the estimates are the same. The same method as in Cases1 and 2 can be used although the situation is even better since each term $I_{k_f , k_g , k_h }$ has the biggest frequencies in the denominator. Taking $s > \frac{1}{2}$ as before we can prove the desired estimate.

\paragraph{\textbf{Case 4: High-Low-High Frequencies}}
This is the most interesting among all the cases, where we'll need to use more tools and not just the Cauchy-Schwarz inequality and the Strichartz estimates for \eqref{1}. Unlike Case 3, the fact that $h$ is frequency localized at a high level makes the handling of the derivative that is introduced by the nonlinearity problematic (so this case is not symmetric to the previous one). The set of frequencies is the following:
$$ \{ k_f , k_g , k_h | k_f \meg 2, k_g \mik k_f - 2, k_f - 1 \mik k_h \mik k_f + 1 \} = J $$
By just applying the Cauchy-Schwarz inequality and applying the Strichartz estimates, we can't eliminate the derivative that was introduced by the nonlinearity. We have to use the functions $w, w_1 , w_2$ to counterbalance the loss of derivative in the numerator. 

We will use Lemma \ref{lem1} to deal with the $H-L-H$ case as it is done in the bottom half of page 11 of \cite{MP}, we include the argument here for completeness. We first express $I_{HLH}$ (which the form of $I$ for this case) with respect to the required localizations:
$$I_{HLH} = \sum_{k_f, k_g, k_h \in J} I_{k_f, k_g, k_h} $$
We further decompose each term in this sum with respect to the $Q_l$ operators:
$$ I_{k_f, k_g, k_h} = \sum_{l_f, l_g, l_h} I_{k_f, k_g, k_h}^{l_f, l_g, l_h} $$
Applying Cauchy-Schwarz and Plancherel, for each term in the left-hand side we have the following bound taking into consideration the localizations that we are imposing:
$$ I_{k_f, k_g, k_h}^{l_f, l_g, l_h} \lesssim \frac{2^{k_f}}{2^{s k_g}} 2^{(-1/2 + 2\ee) l_h} 2^{(-1/2 - \ee) l_f} 2^{(-1/2 - \ee) l_g} \times $$ $$ \times \| P_{k_f} Q_{l_f} f P_{k_g} Q_{l_g} g \|_{L^2_{t,x,y}} \| P_{k_h} Q_{l_h} h \|_{L^2_{t,x,y}} $$
Now we apply Lemma \ref{lem1} and bound the above quantity by the following:
$$ \lesssim 2^{-(s-1/2) k_g} 2^{(-1/2 + 2\ee) l_h} 2^{-\ee l_f} 2^{-\ee l_g} \| P_{k_f} Q_{l_f} f \|_{L^2_{t,x,y}} \| P_{k_g} Q_{l_g} g \|_{L^2_{t,x,y}} \| P_{k_h} Q_{l_h} h \|_{L^2_{t,x,y}} $$
We are ready now to sum over all the indices. Clearly all terms for $l_f, l_g, l_h$ are summed easily, and then we apply Cauchy-Schwarz for $k_g$ to get
$$ I_{HLH} \lesssim \| g \|_{L^2_{t,x,y}} \sum_{k_f, k_h \in J} \| P_{k_f} f \|_{L^2_{t,x,y}} \| P_{k_h} h \|_{L^2_{t,x,y}} $$
which by the fact that $k_f \approx k_h$ is actually
$$\| g \|_{L^2_{t,x,y}} \sum_{k_f} \| P_{k_f} f \|_{L^2_{t,x,y}} P_{k_f} h \|_{L^2_{t,x,y}} $$
and by applying Cauchy-Schwarz with respect to $k_f$ we finally have that:
$$ I_{HLH} \lesssim \| g \|_{L^2_{t,x,y}} \left( \sum_{k_f} \| P_{k_f} f \|_{L^2_{t,x,y}}^2 \right)^{1/2} \left( \sum_{k_f} \| P_{k_f} h \|_{L^2_{t,x,y}}^2 \right)^{1/2} \Rightarrow$$
$$ \Rightarrow I_{HLH} \lesssim \| f \|_{L^2_{t,x,y}} \| g \|_{L^2_{t,x,y}}  \| h \|_{L^2_{t,x,y}} $$
which is the required estimate.
\paragraph{\textbf{Case 5: Low-High-High Frequencies}}
This last case can be treated in the exact same way as Case 4, with the roles of $k_f$ and $k_g$ interchanged.
\end{proof}

\subsection{The Fixed-Point Argument}
Proposition \ref{propbl} of the previous subsection is the main tool for the fixed-point method in $X^{s,b}$ spaces. 

Let us recall some basic facts about these spaces from \cite{T}. First, we define a variation of them, we call them $X^{s,b}_{\dl}$ for some $0 \mik \dl \mik 1$ through the norm:
$$ \| u \|_{X^{s,b}_{\dl}} = \inf_{v(t) = u(t), t\in [0,\dl]} \| v \|_{X^{s,b}} $$
We have the following theorem.
\begin{thm}\label{xsbthm}
The $X^{s,b}_{\dl}$ spaces (for some $\dl \in (0,1)$) have the following properties: 
$$ 1. \quad \| \chi(t) e^{itNV} f \|_{X^{s,b}_{\dl}} \lesssim \| f \|_{H^s (\rr^2)} \mbox{ for any $s,b \in \rr$} $$
$$ 2. \quad \left\| \chi(t) \int_0^t e^{i(t-t')NV} F(t') dt' \right\|_{X^{s,b}_{\dl}} \lesssim \| F \|_{X^{s,b-1}_{\dl}} $$
 for any $s \in \rr$ and $-\frac{1}{2} < b-1 \mik 0$.
$$ 3. \quad \| u \|_{X^{s,b'}_{\dl}} \lesssim \dl^{b-b'} \| u \|_{X^{s,b}_{\dl}} \mbox{ for any $s\in \rr$ and $-\frac{1}{2} < b' < b < \frac{1}{2} $} $$
where $\chi$ is a $C^{\infty}_0 (\rr)$ function that is equal to 1 in $[-1,1]$.
\end{thm}
Now we are ready to give the proof of Theorem \ref{wpnv}
\begin{proof}[Proof of Theorem \ref{wpnv}]
We write Duhamel's formula for \eqref{2}:
$$ u(t,x,y) = \chi(t) e^{itNV} \ph(x,y) + \chi(t) \int_0^t e^{i(t-t')NV} NL(u)(t' , x,y) dt' $$
with a function $\chi$ as before. We evaluate $u$ in the $X^{s, 1/2 + \ee}_{\dl}$ norm for $s > \frac{1}{2}$, some $\ee > 0$, and some $0 < \dl < 1$ small (to be chosen later), and we apply successively 1, 2, 3 of Theorem \ref{xsbthm} and, finally, the estimate \eqref{nl} (which is a consequence of Proposition \ref{propbl}):
$$ \| u \|_{X^{s, 1/2 + \ee}_{\dl}} \lesssim \| \phi \|_{H^s(\rr^2)} + \left\| \chi(t) \int_0^t e^{i(t-t')NV} NL(u)(t' ,x,y) dt' \right\|_{X^{s,1/2+\ee}_{\dl}}  \lesssim $$ $$ \lesssim \| \phi \|_{H^s(\rr^2)} + \| NL(u) \|_{X^{s, -1/2 + \ee}_{\dl}} \lesssim $$ $$ \lesssim \| \phi \|_{H^s(\rr^2)} + \dl^{\ee} \| NL(u) \|_{X^{s, -1/2 + 2\ee}_{\dl}} \lesssim $$ $$ \lesssim \| \phi \|_{H^s(\rr^2)} + \dl^{\ee} \| u \|_{X^{s, 1/2 + \ee}_{\dl}}^2 $$
Since $\ee > 0$, we can apply the fixed-point argument by choosing an appropriate $\dl = \dl(\| \phi \|_{H^s (\rr^2)})$ and this finishes the proof of the Theorem.
\end{proof}

\section{Well-Posedness for the modified Novikov-Veselov equation}
We will follow the same lines for the the proof of Theorem \ref{wpmnv}. The fixed-point argument will take place in the same $X^{s,b}$ spaces.
\subsection{Trilinear Estimates}
In this situation we will need a trilinear estimate in $X^{s,b}$ spaces. It has the following form:
\begin{prop}\label{proptrl}
 Let $s > 1$ and define the quantity
 $$mNL_1 (u,v,w) = \frac{3}{4} \partial u \bar{\partial}^{-1} \partial (v \bar{w} ) $$
 Then we have the inequality:
\begin{equation}\label{xsbm}
 \| mNL_1 (u,v,w) \|_{X^{s,-1/2+2\ee}} \lesssim \| u \|_{X^{s,1/2+\ee}} \| v \|_{X^{s,1/2+\ee}} \| w \|_{X^{s,1/2+\ee}}
\end{equation}
for any $\ee > 0$.
\end{prop}
\begin{rem}
 As before, the proof should apply as well to the other nonlinearities (defined analogously) $mNL_{2,3,4} (u,v,w)$, so that we'll have in the end
 $$ \| mNL_1 (u) + mNL_2 (u) + mNL_3 (u) + mNL_4 (u) \|_{X^{s,-1/2+2\ee}} \lesssim \| u \|_{X^{s,1/2+\ee}}^3 $$
 \end{rem}
\begin{proof}
As in the bilinear situation, we'll break the proof into several cases by taking again Fourier localizations (with respect to the space variables) for the different functions involved.

First we rewrite the estimate in its dual form as before. So \eqref{xsbm} is equivalent to
$$II = \int_{\mathbb{R}^9} K(\tau,\tau_1,\tau_2,\xi,\xi_1,\xi_2,\mu,\mu_1,\mu_2) \tilde{e} (\tau_1,\xi_1,\mu_1) \tilde{f} (\tau_2-\tau_1, \xi_2-\xi_1, \mu_2-\mu_1) \times $$ $$\times\tilde{g} (\tau-\tau_2,\xi-\xi_2,\mu-\mu_2) \tilde{h}(\tau,\xi,\mu) d\tau d\tau_1 d\tau_2 d\xi d\xi_1 d\xi_2 d\mu d\mu_1 d\mu_2 \lesssim $$ $$ \lesssim \| e \|_{L^2_{\tau,\xi,\mu}} \| f \|_{L^2_{\tau,\xi,\mu}} \| g \|_{L^2_{\tau,\xi,\mu}} \| h \|_{L^2_{\tau,\xi,\mu}} $$ 
where 
$$K(\tau,\tau_1,\tau_2,\xi,\xi_1,\xi_2,\mu,\mu_1,\mu_2) = $$ $$\dfrac{|(\xi,\mu)|\langle (\xi,\mu) \rangle^s}{\langle w \rangle^{\frac{1}{2}-2\ee} \langle (\xi_1,\mu_1) \rangle^s \langle w_1 \rangle^{\frac{1}{2}+\ee} \langle (\xi-\xi_1, \mu-\mu_1 ) \rangle^s \langle w_{3} \rangle^{\frac{1}{2}+\ee}  \langle (\xi_2 - \xi_1, \mu_2 -\mu_1 ) \rangle^s \langle w_{4} \rangle^{\frac{1}{2}+\ee} }  $$
for
$$ w_{3} = w(\tau_2 - \tau_1, \xi_2-\xi_1,\mu_2-\mu_1), \quad \quad w_{4} = w(\tau-\tau_2, \xi-\xi_2,\mu-\mu_2) $$
and $\tilde{e}, \tilde{f}, \tilde{g}, \tilde{h}$ defined in an analogous way to the bilinear case.

Again we use $k_e, k_f, k_g, k_h$ to denote the frequencies corresponding to the functions $e,f,g,h$.

\paragraph{\textbf{Case 1: High-High-High-High and Low-Low-Low-Low Interactions}}
These two situations can be treated in the same as for the bilinear estimate. Assume that
$$k_e \approx k_f \approx k_g \approx k_h \approx k $$
where $k$ is either $\mik 1$ or $>>1$. We apply Cauchy-Schwarz and the $L^6$ Strichartz estimate (which ends up in a $1/6$ loss of derivative) in its $X^{s,b}$ form:
$$ \left\| \mathcal{F}^{-1} \left( \dfrac{\widetilde{P_{k_f} f}}{\langle w \rangle^{1/2+\ee}} \right) \right\|_{L^6_{t,x,y}} \lesssim 2^{\frac{k_f}{6}} \| f \|_{L^2_{t,x,y}} $$
We have:
$$ II_{k_e, k_f, k_g, k_h} \lesssim \dfrac{2^{(s+1)k}}{2^{3sk}} \| P_{k_h} h \|_2 \times $$ $$ \times \left\| \mathcal{F}^{-1} \left( \dfrac{\widetilde{P_{k_e} e}}{\langle w \rangle^{1/2+\ee}} \right) \mathcal{F}^{-1} \left( \dfrac{\widetilde{P_{k_f} f}}{\langle w \rangle^{1/2+\ee}} \right) \mathcal{F}^{-1} \left( \dfrac{\widetilde{P_{k_g} g}}{\langle w \rangle^{1/2+\ee}} \right) \right\|_2 \lesssim $$ $$ \lesssim \dfrac{2^{(s+3/2)k}}{2^{3sk}} \| e \|_2 \| f \|_2 \| g \|_2 \| h \|_2 $$
Now we sum all the pieces in the High-High-High-High case and we have:
$$ II_{HHHH} = \sum_{k >> 1} II_{k_f, k_g, k_h} \sum_{k >> 1} \dfrac{2^{(s+3/2)k}}{2^{3sk}} \| P_{k_e} e \|_2 \| P_{k_f} f \|_2 \| P_{k_g} g \|_2 \| P_{k_h} h \|_2 $$
and by assuming that 
$$ s+\frac{3}{2} - 3s < 0 \Rightarrow s > \frac{3}{4} $$
we can apply Cauchy-Schwarz for every piece and conclude that:
$$ II_{HHHH} \lesssim \| e \|_2 \| f \|_2 \| g \|_2 \| h \|_2 $$
Note that in this particular case we have a better bound for $s$ than the one stated in the Theorem.

\paragraph{\textbf{Case 2: Low-High-High-High and High-Low-High-High Interactions}}
We will treat only the case of High-Low-High-High Interactions since the case of Low-High-High-High Interactions is easier. This is because $h$ is localized at a frequency level that is smaller than all other functions, so the derivative coming from the nonlinearity can be easily balanced.

As before, we decompose $II$ with respect to the $Q$ operator as well and we have:
$$ II_{k_e, k_f, k_g, k_h} = \sum_{l_e, l_f, l_g, l_h} II_{k_e, k_f, k_g, k_h}^{l_e, l_f, l_g, l_h} $$
We assume the following relation for the frequencies:
$$ k_h \approx k_e \approx k_f >> k_g $$
We apply Cauchy-Schwarz and 2) of Lemma \ref{lem2} with $e$ in the role of $h$, and we have that
$$ II_{k_e, k_f, k_g, k_h}^{l_f, l_g, l_h} \lesssim $$ $$ \lesssim \dfrac{2^{(s+1)k_h} 2^{\frac{k_g}{2}}}{2^{s k_e} 2^{s k_f} 2^{s k_g}} \dfrac{2^{\frac{l_e}{2}} 2^{\frac{l_f}{2}} 2^{\frac{l_g}{2}}}{2^{(1/2-2\ee)l_h} 2^{(1/2+\ee)l_e} 2^{(1/2+\ee)l_f}  2^{(1/2+\ee)l_g}} \times $$ $$ \times \| P_{k_e} Q_{l_e} e \|_2 \| P_{k_f} Q_{l_f} f \|_2 \| P_{k_g} Q_{l_g} g \|_2 \| P_{k_h} Q_{l_h} h \|_2$$
By our condition on the frequencies we have that 
$$\dfrac{2^{(s+1)k_h} 2^{\frac{k_g}{2}}}{2^{s k_e} 2^{s k_f} 2^{s k_g}} \approx 2^{(s+1-2s) k_h} 2^{(1/2-s) k_g} $$
We impose the following condition on $s$:
$$ s+1-2s < 0 \Rightarrow s > 1 $$
And now we can sum first all $l$ frequencies using Cauchy-Schwarz as: 
$$ \sum_{l_e, l_f, l_g, l_h} II_{k_e, k_f, k_g, k_h}^{l_e, l_f, l_g, l_h} \lesssim $$ $$ \lesssim \sum_{l_e, l_f, l_g, l_h} 2^{(s+1-2s) k_h} 2^{(1/2-s) k_g} \dfrac{1}{2^{(1/2-2\ee)l_h} 2^{\ee l_e} 2^{\ee l_f}  2^{\ee l_g}} \times $$ $$ \times \| P_{k_e} Q_{l_e} e \|_2 \| P_{k_f} Q_{l_f} f \|_2 \| P_{k_g} Q_{l_g} g \|_2 \| P_{k_h} Q_{l_h} h \|_2 \lesssim $$ $$ \lesssim 2^{(s+1-2s) k_h} 2^{(1/2-s) k_g}  \| P_{k_e} e \|_2 \| P_{k_f} f \|_2 \| P_{k_g} g \|_2 \| P_{k_h} h \|_2 $$
Using our condition on $s$ we can apply Cauchy-Schwarz for the $k$ frequencies as well and finally conclude that:
$$ II_{HLHH} = \sum_{k_h \approx k_e \approx k_f >> k_g} \sum_{l_e, l_f, l_g, l_h} II_{k_e, k_f, k_g, k_h}^{l_e, l_f, l_g, l_h} \lesssim $$ $$ \lesssim \sum_{k_h \approx k_e \approx k_f >> k_g} 2^{(s+1-2s) k_h} 2^{(1/2-s) k_g}  \| P_{k_e} e \|_2 \| P_{k_f} f \|_2 \| P_{k_g} g \|_2 \| P_{k_h} h \|_2 \Rightarrow $$ $$ \Rightarrow II_{HLHH} \lesssim \| e \|_2 \| f \|_2 \| g \|_2 \| h \|_2 $$
There are more situations where the High-Low-High Interactions can be imposed differently on $e,f,g$ but they are all symmetric, so the same proof works.

\paragraph{\textbf{Case 3: Low-High-Low-High and High-Low-High-Low Interactions}}
Again as in Case 2, we won't deal with the Low-High-Low-High Interactions since they can be treated in the same way as the interactions in Case 2 (they are actually easier since the derivative coming from the nonlinearity can be easily balanced as it is on the ``low level''). 

For the High-Low-High-Low interactions we consider the following condition on frequencies:
$$ k_h \approx k_f >> k_e \approx k_g $$
We use again the same decomposition as in Case 2:
$$ II_{k_e, k_f, k_g, k_h} = \sum_{l_e, l_f, l_g, l_h} II_{k_e, k_f, k_g, k_h}^{l_e, l_f, l_g, l_h} $$
and apply Cauchy-Schwarz for each piece, but now combined with 1) of Lemma \ref{lem1}:
$$II_{k_e, k_f, k_g, k_h}^{l_e, l_f, l_g, l_h} \lesssim $$ $$ \lesssim \dfrac{2^{(s+1)k_h} 2^{\frac{3 k_g}{2}}}{2^{s k_e} 2^{(s+1) k_f} 2^{s k_g}} \dfrac{2^{\frac{l_e}{2}} 2^{\frac{l_f}{2}} 2^{\frac{l_g}{2}}}{2^{(1/2-2\ee)l_h} 2^{(1/2+\ee)l_e} 2^{(1/2+\ee)l_f}  2^{(1/2+\ee)l_g}} \times $$ $$ \times \| P_{k_e} Q_{l_e} e \|_2 \| P_{k_f} Q_{l_f} f \|_2 \| P_{k_g} Q_{l_g} g \|_2 \| P_{k_h} Q_{l_h} h \|_2$$  
Using the condition on the $k$ frequencies we have:
$$ \dfrac{2^{(s+1)k_h} 2^{\frac{3 k_g}{2}}}{2^{s k_e} 2^{(s+1) k_f} 2^{s k_g}} \approx  2^{(3/2-2s) k_g} $$
Imposing the condition 
$$ \frac{3}{2} - 2s < 0 \Rightarrow s > \frac{3}{4} $$
we can apply Cauchy-Schwarz to add up all the $k$ and $l$ frequency pieces as in Case 2 and get again the required estimate: 
$$ II_{HLHL} \lesssim \| e \|_2 \| f \|_2 \| g \|_2 \| h \|_2 $$
Note that in this case as well we have a better bound for $s$ than the one stated in the Theorem.

\end{proof}
\begin{rem}
 Note that there is no High-Low-Low-Low interactions case. Also note that there are other more complicated cases with interactions on three different levels but they can be treated in the similar ways as the interactions above. 
\end{rem}
\subsection{The Fixed-Point Argument}
Applying Proposition \ref{proptrl} on a solution of \ref{mnv} given by the Duhamel formula, and using the properties of the $X^{s,b}$ spaces from Theorem \ref{xsbthm} we have the estimate:
$$ \| u \|_{X^{s, 1/2 + \ee}_{\dl}} \lesssim \| \phi \|_{H^s(\rr^2)} + \dl^{\ee} \| u \|_{X^{s, 1/2 + \ee}_{\dl}}^3 $$
We can see then from this that the proof of Theorem \ref{wpmnv} can be given in the exact same manner as the one of Theorem \ref{wpnv}, so it won't be repeated here.

\section{Ill-Posedness Issues for the Novikov-Veselov equation}
In this final section we will prove Theorem \ref{ipnv}. This will be done by proving a failure of differentiability at the origin of the data-to-solution map of \eqref{1}.
\begin{thm}\label{dtsm}
 For any $s < -1$, there is no $T > 0$ such that the data-to-solution map of \eqref{1} 
 $$ NV(t) : \phi \rightarrow u(t), \quad t\in [0,T] $$ is $C^2$ at 0 as a map from $\dot{H}^s (\mathbb{R}^2 )$ to $\dot{H}^s (\mathbb{R}^2 )$.
\end{thm}
Deriving Theorem \ref{ipnv} from Theorem \ref{dtsm} is a standard fact and won't be shown here, the interested reader can take a look at the proof of the analogous Theorem 5.2 for the KP-I equation in \cite{MST1}. 

Before giving a proof of Theorem \ref{dtsm} though, we will record the failure of a bilinear estimate for the (related to the homogeneous $\dot{H}^s$ spaces) $\dot{X}^{s,b}$ spaces that are defined as
\begin{displaymath}
\dot{X}^{s,b} = \left\{ u \in L^2_{t,x} | |(\xi, \mu)|^s \left\langle \tau - \frac{1}{4}\xi^3 + \frac{3}{4}\xi \mu^2 \right\rangle^b \tilde{u} (\tau, \xi, \mu) \in L^2_{\tau, \xi, \mu} \right\} \mbox{ with norm }
\end{displaymath}
\begin{displaymath}
\| u \|_{\dot{X}^{s,b}} = \left\| |(\xi, \mu)|^s \left\langle \tau - \frac{1}{4} \xi^3 + \frac{3}{4} \xi \mu^2 \right\rangle^b \tilde{u} (\tau, \xi, \mu) \right\|_{L^2_{\tau, \xi, \mu}}
\end{displaymath}
\begin{prop}\label{fbl}
 The inequality
 $$ \| NL_1 (u,v) \|_{\dot{X}^{s,-b'}} \lesssim \| u \|_{\dot{X}^{s,b}} \| v \|_{\dot{X}^{s,b}} $$
 fails for any $b,b' \in \rr$ and $s < -1$.
\end{prop}
\begin{rem}
 The same holds true for $NL_2 (u,v)$ which has as a result the failure of the control in $X^{s,b}$ of the nonlinear part. Note also that this result is natural since scaling considerations indicate that the critical Sobolev space for \eqref{1} is $\dot{H}^{-1} (\rr^2)$. To be a bit more precise a scaled solution for the equation is the following one:
 $$ u_{\lambda} (t,x,y) = \frac{1}{\lambda^2} u \left( \frac{t}{\lambda^3}, \frac{x}{\lambda}, \frac{y}{\lambda} \right) $$
 and then we notice that we have:
 $$ \phi_{\lambda} (x,y) = \frac{1}{\lambda^2} \phi \left( \frac{x}{\lambda}, \frac{y}{\lambda} \right) \Rightarrow \| \phi_{\lambda} \|_{\dot{H}^{-1} (\rr^2)} = \| \phi \|_{\dot{H}^{-1} (\rr^2)} $$ 
\end{rem}
\begin{rem}
It should be expected that a similar result is true for the modified Novikov-Veselov equation \eqref{mnv}. In this situation we should have that the data-to-solution map fails to be $C^3$ at the origin with respect to the topology of $\dot{H}^s (\rr^2)$ for any $s < 0$, as we can see according to scaling considerations. The solution of \eqref{mnv} remains invariant under the following transformation:
$$ u_{\lambda} (t,x,y) = \frac{1}{\lambda} u \left( \frac{t}{\lambda^3}, \frac{x}{\lambda}, \frac{y}{\lambda} \right) $$
and in this case we have:
 $$ \phi_{\lambda} (x,y) = \frac{1}{\lambda} \phi \left( \frac{x}{\lambda}, \frac{y}{\lambda} \right) \Rightarrow \| \phi_{\lambda} \|_{L^2 (\rr^2)} = \| \phi \|_{L^2 (\rr^2)} $$ 
 \end{rem}
Proposition \ref{fbl} won't imply directly Theorem \ref{dtsm}, but the counterexample that will cause the failure of the bilinear $\dot{X}^{s,b}$ estimate will be used to cause the failure of the differentiability at the origin of the data-to-solution map.
\begin{proof}[Proof of Proposition \ref{fbl}]
 We employ again the dual formulation of the $\dot{X}^{s,b}$ estimate. We rewrite it here for the convenience of the reader:
 $$ \int_{\mathbb{R}^6} K_0(\tau,\tau_1,\xi,\xi_1,\mu.\mu_1) \tilde{f} (\tau_1,\xi_1,\mu_1) \tilde{g} (\tau-\tau_1, \xi-\xi_1, \mu-\mu_1) \tilde{h} (\tau,\xi,\mu) d\tau d\tau_1 d\xi d\xi_1 d\mu d\mu_1 $$ $$ \lesssim \| f \|_{L^2_{\tau,\xi,\mu}} \| g \|_{L^2_{\tau,\xi,\mu}} \| h \|_{L^2_{\tau,\xi,\mu}} $$ where the function $K_0$ is defined as $$ K_0(\tau,\tau_1,\xi,\xi_1,\mu.\mu_1) = \dfrac{|(\xi,\mu)||(\xi,\mu)|^s}{\langle w \rangle^{b'} |(\xi_1,\mu_1)|^s \langle w_1 \rangle^{b} |(\xi-\xi_1, \mu-\mu_1 )|^s \langle w_2 \rangle^{b} } $$ 
   
\begin{rem}
We have to note that the definition of $K_0$ is not strictly correct since we would like to bound the quantity defined above by below. The fraction that shows up because of the nonlinearity is the following: 
$$ \dfrac{(i\xi+\mu)(i(\xi-\xi_1) + (\mu-\mu_1))}{i(\xi-\xi_1) - (\mu-\mu_1)} $$
We can make the following changes now:
$$ i\xi + \mu = |i\xi + \mu| e^{iH(\xi,\mu)} $$
for some function $H$. Moreover by noticing that 
$$ \left| \dfrac{i(\xi-\xi_1) + (\mu-\mu_1)}{i(\xi-\xi_1) - (\mu-\mu_1)} \right| = 1 $$
we can write in the end:
$$ \dfrac{(i\xi+\mu)(i(\xi-\xi_1) + (\mu-\mu_1))}{i(\xi-\xi_1) - (\mu-\mu_1)}  = |i\xi + \mu| e^{iH(\xi,\mu)} e^{i\tilde{H}(\xi-\xi_1, \mu-\mu_1)}$$
for some other function $\tilde{H}$. These two new terms are just phase changes that can be absorbed in the definitions of the functions $h$ and $g$ respectively. The final outcome is not affected since they don't influence in any significant way the $L^2$ norms of these functions. So finally, our definition of $K_0$ is good enough for our purposes. 
\end{rem}
The functions $f,g,h$ are given as before as follows (note that the absolute values are there as before, but in this case too we can argue as in the remark above):
$$ \tilde{f} (\tau_1,\xi_1,\mu_1) = |\tilde{u} (\tau_1,\xi_1,\mu_1) \langle w_1 \rangle^{b} |(\xi_1,\mu_1)|^s |, $$ $$ \tilde{g} (\tau-\tau_1, \xi-\xi_1, \mu-\mu_1) = |\tilde{v}(\tau-\tau_1, \xi-\xi_1, \mu-\mu_1) \langle w_2 \rangle^{b} |(\xi-\xi_1, \mu-\mu_1)|^s $$ $$ \tilde{h} (\tau, \xi, \mu) = |\tilde{h'} (\tau,\xi,\mu) \langle w \rangle^{b'} |(\xi,\mu)|^{-s}| $$

We define now the functions $f,g,h$ that will show the failure of the bilinear estimate.
\begin{defn}
The function $f,g,h$ that will be used from now are defined as follows
$$ \widetilde{f}(\tau_1, \xi_1, \mu_1) = \begin{cases} 1 \mbox{ for $\xi_1 \in [-N-c, N], \mu_1 \in [-2c, -c], w_1 \in [0,1]$} \\
0 \mbox{ otherwise}
\end{cases} $$ 

$$ \widetilde{g}(\tau', \xi', \mu' ) = \begin{cases} 1 \mbox{ for $\xi' \in [N+2c, N+3c], \mu' \in [3c, 4c], |w_2| \mik 2 + Cc N^2$} \\
0 \mbox{ otherwise}
\end{cases}$$ 
where $\tau' = \tau-\tau_1, \xi' = \xi-\xi_1, \mu' = \mu-\mu_1$

$$ \widetilde{h}(\tau, \xi,\mu) = \begin{cases} 1 \mbox{ for $\xi \in [c, 3c]. \mu \in [c, 3c], w \in [0,1]$} \\
0 \mbox{ otherwise}        
\end{cases}
$$ 
for some $0 < c \approx \frac{1}{N^2} << 1$, i.e. $N >> 1$, and any $\ee > 0$.  
\end{defn}
Note that with this choice of functions, the ``resonant'' function satisfies the following bound in their supports:
$$ |R(\xi_1, \xi-\xi_1, \mu_1, \mu-\mu_1)| \approx cN^2 \approx const. $$
by the choice of $c$. We can now compute for $(\tau, \xi, \mu) \in supp(\widetilde{h})$: $$ \left( \dfrac{\widetilde{f}}{\langle w \rangle^b} \ast \dfrac{\widetilde{g}}{\langle w \rangle^b} \right) (\tau, \xi, \mu) = \int_{supp(\widetilde{f})} \dfrac{1}{\langle w_1 \rangle^b \langle w_2 \rangle^b} d\tau_1 d\xi_1 d\mu_1 = $$ $$ = \int_{-N-c}^{-N} \int_{-2c}^{-c} \int_{w_1 \in [0,1]} \dfrac{d\tau_1 d\xi_1 d\mu_1}{\langle w_1 \rangle^b \langle w - w_1 + R\rangle^b} \gtrsim $$ $$ \gtrsim\int_{-N-c}^{-N} \int_{-2c}^{-c} \int_{w_1 \in [0,1]} \dfrac{d\tau_1 d\xi_1 d\mu_1}{\langle w_1 \rangle^b \langle 1 + R\rangle^b} \gtrsim $$ $$ \gtrsim \int_{-N-c}^{-N} \int_{-2c}^{-c}  \dfrac{d\xi_1 d\mu_1}{\langle 1 + R\rangle^b} \gtrsim c^2 $$ 
where the first from the inequalities holds as $|w - w_1| \lesssim 1$, the second one by Fubini and the last one by noticing that since we work in the support of $f$ and $g$, we have that $$ R \approx cN^2 \approx 1 \Rightarrow \dfrac{1}{1 + |R|^b} \approx \dfrac{1}{1 + |cN^2 |^b } \approx const. $$ 

On the other hand we have the following:
$$K_0(\tau,\tau_1,\xi,\xi_1,\mu.\mu_1) \gtrsim \dfrac{N^{-4s}}{N^2 \langle w \rangle^{b'} \langle w_1 \rangle^b \langle w_2 \rangle^b} = $$ $$ = \dfrac{N^{-4s-2}}{\langle w \rangle^{b'} \langle w_1 \rangle^b \langle w_2 \rangle^b} = $$
 So finally using all previous estimates, we note now that the integral of interest is bounded below by: 
 $$ c^2 N^{4s-2} \int_{c}^{3c} \int_{c}^{3c} \int_{w\in[0,1]} \dfrac{1}{ \langle w \rangle^{b'}} d\tau d\xi d\mu \gtrsim c^4 N^{-4s-2} $$ 

By the definitions of $f,g,h$ we can compute directly their $L^2$ norms: 
$$ \| f \|_{L^2} = c, \quad \| g \|_{L^2} = c (2 + Cc N^2 )^{1/2}, \quad \| h \|_{L^2} = c $$ 
So if the inequality that was initially stated was true, we would have: 
$$ c^4 N^{-4s-2}\lesssim c^3 (2 + Cc N^2 )^{1/2}$$ 
and since we assume that $cN^2 \approx 1$ we would actually have that: 
$$ N^{-4s-4} \lesssim 1 $$ 
which gives us a contradiction assuming that
$$ -4s - 4 > 0 \Leftrightarrow s < -1 $$
since $N >> 1$. 
\end{proof}
We are now ready to give the proof of Theorem \ref{dtsm}.
\begin{proof}[Proof of Theorem \ref{dtsm}]
 We consider a parameter $\ee > 0$ and the equation
 \begin{equation}\label{4}
 \left\{\begin{aligned}
        \partial_t u + \partial^3 u +\bar{\partial}^3 u + NL_1 (u) + NL_2 (u) = 0 \\
        u(0,x) = \ee \phi(x)\
        \end{aligned}
 \right.
\end{equation}
For \eqref{4} a standard computation shows that
$$ \left.\dfrac{\partial^2 u}{\partial \ee}\right|_{\ee = 0} (t,x,y) := u_2 (t,x,y) = \int_0^t e^{i(t-s)NV} NL(e^{isNV} \phi , e^{isNV} \phi) ds $$
which is the second derivative of the data-to-solution map $NV$ for \eqref{2} evaluated at 0, where $NL = NL_1 + NL_2$. Note that we'll show the computations only for $NL_1$, the case of $NL_2$ is similar. 

If $NV$ was $C^2$ at the origin, then we would have the following inequality for $u_2$:
\begin{equation}\label{5}
\| u_2 (t,x,y) \|_{\dot{H}^s (\mathbb{R}^2 )} \lesssim \| \phi \|^2_{\dot{H}^s (\mathbb{R}^2 )}
\end{equation}
Hence our goal in order to complete the proof of Theorem \ref{ipnv} is to show the failure of \eqref{5}. We proceed by making a choice for $\phi$ based on the computations in the proof of Proposition \ref{fbl}. Specifically we choose a function such that it consists of a part that is supported  $supp(\widetilde{f})$ and another part that is supported in $supp(\widetilde{h})$ (as they were given in the proof of Proposition \ref{fbl}). 
\begin{defn}
 From now on, the function $\phi$ will have the following form (we define it through its Fourier transform):
 $$ \hat{\phi} (\xi ,\mu) = c^{-1} N^{-s} \chi_1 (\xi,\mu) + c^{-1} N^{-s} \chi_2 (\xi , \mu) $$ $$ \mbox{ where $\chi_1$ is the indicator function of $D_1 = [-N-c , -N] \times [-2c, -c]$ } $$ $$ \mbox{ and $\chi_2$ of $D_2 = [N+2c, N+3c] \times [3c, 4c] $ } $$
\end{defn}
As before we choose
$$ cN^2 \approx 1 $$

In order to use this definition appropriately we present some formulas first. The analogue of Lemma 4, page 376 of \cite{MST1} reads as follows for the Novikov-Veselov equation (its proof is the same): $$ \int_0^t e^{i(t-s)PV} F(s, x, y) ds = $$ $$ = const. \int_{\mathbb{R}^3} e^{it(\xi^3 - \xi^2 \mu) + ix\xi + iy\mu} \dfrac{e^{it(\tau - \xi^3 + \xi^2 \mu)} - 1}{\tau - \frac{1}{4}\xi^3 + \frac{3}{4}\xi^2 \mu} \tilde{F} (\tau,\xi,\mu) d\tau d\xi d\mu $$

After some computations we arrive at the following formula:
$$ u_2 (t,x,y) = const. \int_{\mathbb{R}^4} (i\xi + \mu) e^{it(\xi^3 - \xi^2 \mu) + ix\xi + iy\mu} \dfrac{e^{itR (\xi_1, \xi-\xi_1, \mu_1, \mu-\mu_1)} - 1}{R (\xi_1, \xi-\xi_1, \mu_1, \mu-\mu_1)} \times $$ $$ \quad \quad \quad \quad \quad \times \dfrac{i(\xi - \xi_1 ) + (\mu -\mu_1 )}{i(\xi - \xi_1 ) - (\mu - \mu_1 )} \hat{\phi} (\xi_1 , \mu_1 ) \hat{\phi} (\xi - \xi_1 , \mu - \mu_1 ) d\xi d\mu d\xi_1 d\mu_1 $$

By the definition of $\phi$,  the formula for $u_2$ can be broken into three parts, according to where $(\xi_1,\mu_1)$ and $(\xi-\xi_1, \mu-\mu_1)$ belong to. In two of these parts $(\xi_1 ,\mu_1)$ and $(\xi-\xi_1, \mu-\mu_1)$ belong to the same set and in the third one they belong to different sets. Denoting the function inside the integral for $u_2$ by $\Phi$ we have that $$ u_2 (t,x,y) = I(t,x,y) + II(t,x,y) + III(t,x,y) \mbox{ where } $$ $$  I(t,x,y) = \frac{const.}{c^2 N^{2s}} \int_{(\xi_1,\mu_1) \in D_1, (\xi-\xi_1, \mu-\mu_1) \in D_1} \Phi d\xi d\mu d\xi_1 d\mu_1 $$ $$ II(t,x,y) =   \frac{const.}{c^2 N^{2s}} \int_{(\xi_1,\mu_1) \in D_2, (\xi-\xi_1, \mu-\mu_1) \in D_2} \Phi d\xi d\mu d\xi_1 d\mu_1  $$ $$III(t,x,y) =  \frac{const.}{c^2 N^{2s} } \int_{(\xi_1,\mu_1) \in D_1, (\xi-\xi_1, \mu-\mu_1) \in D_2} \Phi d\xi d\mu d\xi_1 d\mu_1 + $$ $$ + \frac{const.}{cdN^s } \int_{(\xi_1,\mu_1) \in D_2, (\xi-\xi_1, \mu-\mu_1) \in D_1} \Phi d\xi d\mu d\xi_1 d\mu_1  $$
First we give upper bounds for $I$. We take absolute values inside and we get rid of the exponential, also we note that for $(\xi_1 ,\mu_1), (\xi-\xi_1, \mu-\mu_1) \in D_1$ we have that:
$$ |\xi| \approx |\xi-\xi_1| \approx |\xi_1| \approx N, \quad |\mu| \approx |\mu-\mu_1| \approx |\mu_1| \approx c $$
which implies according to our choice of $c$ that
$$ |R (\xi_1, \xi-\xi_1, \mu_1, \mu-\mu_1)| \approx N^3 $$
Now we can compute the following:
$$ \| I \|_{\dot{H}^s (\rr^2)} \lesssim \dfrac{N^{s+1} c}{N^{2s} N^3} = N^{-s-4} $$
Similarly we can get the following bound for $II$:
$$ \| II \|_{\dot{H}^s (\rr^2)} \lesssim \dfrac{N^{s+1} c}{N^{2s} N^3} = N^{-s-4} $$
We now turn to $III$ and for convenience we break it into two parts, $IV$ and $V$. Taking the Fourier transform of $IV$ as $(x,y) \rightarrow (\xi,\mu)$ we note that $$ \mathcal{F} IV (t,\xi,\mu) =  \dfrac{const. (i\xi + \mu) e^{it(\xi^3 - \xi^2 \mu)}}{cdN^s} \int_{(\xi_1,\mu_1) \in D_1, (\xi-\xi_1, \mu-\mu_1) \in D_2} \times $$ $$ \times \dfrac{e^{itR (\xi_1, \xi-\xi_1, \mu_1, \mu-\mu_1)} - 1}{R (\xi_1, \xi-\xi_1, \mu_1, \mu-\mu_1)} \dfrac{i(\xi - \xi_1 ) + (\mu -\mu_1 )}{i(\xi - \xi_1 ) - (\mu - \mu_1 )} d\xi_1 d\mu_1 $$ 

For $R$ we use the computations of the previous section to show that for $(\xi_1,\mu_1) \in D_1$ and  $(\xi-\xi_1, \mu-\mu_1) \in D_2$ we have the estimate: 
$$|R (\xi_1, \xi-\xi_1, \mu_1, \mu-\mu_1)| \approx cN^2 \approx 1$$ 
Using our definition of $c$ we can see that
$$ \left| \dfrac{e^{itR} - 1}{R} \right| \gtrsim 1 $$ 

Note that we have to deal also with the other fraction inside $\mathcal{F} IV$, but this is not really a problem, since once we take the modulus of it (as we will in order to compute the $\dot{H}^s$ norm of $IV$) we note that this fraction has modulus 1, so it can be seen as a phase change which won't affect in any significant way the computation of the measure of the set that we are interested in (for more on this, see the related remark in the proof of Proposition \ref{fbl}).

So finally we have: 
$$ \| IV \|_{\dot{H}^s (\mathbb{R}^2)} \gtrsim \dfrac{ccc^s}{N^{2s}} \approx N^{-4s-4} $$ 
The same estimate holds for $V$, so in the end we can state that:
$$ \| III \|_{\dot{H}^s (\mathbb{R}^2)} \gtrsim \dfrac{ccc^s}{N^{2s}} \approx N^{-4s-4} $$
If \eqref{5} was true, then as $\| \phi \|_{\dot{H}^s (\rr^2)} \approx 1$, we would have that:
$$ 1 \approx \| \phi \|_{\dot{H}^s (\rr^2)} \gtrsim  \| u_2 (t,x,y) \|_{\dot{H}^s (\mathbb{R}^2 )} \gtrsim $$ $$ \gtrsim \| III \|_{\dot{H}^s (\mathbb{R}^2)} - \| II \|_{\dot{H}^s (\mathbb{R}^2)} - \| I \|_{\dot{H}^s (\mathbb{R}^2)} \gtrsim N^{-4s-4} - N^{-s-4} $$ $$ \Rightarrow N^{-4s-4} \lesssim 1 + N^{-s-4} $$
We use now that $s < -1$. We consider two cases.

\textbf{(i) $-4 \mik s < -1$: } In this case we have that $-s-4 \mik 0$. This implies that $N^{-4s-4} \lesssim 1$ which is a contradiction since $N >> 1$ and $-4s-4 > 0$.

\textbf{(ii) $s < -4$: } In this case $-s-4 > 0$ which implies that 
$$ N^{-4s-4} \lesssim N^{-s-4} \Rightarrow N^{-3s} \lesssim 1 $$
a contradiction again because $N >> 1$ and $-3s > 0$.

\end{proof}

\end{document}